\documentclass[reqno]{amsart}
\usepackage{amsmath, amsfonts, amsthm, amssymb, setspace, textcomp,
	fullpage, bbm}
\newtheorem{theorem}{Theorem}[section]

\theoremstyle{definition}

\theoremstyle{remark}

\numberwithin{equation}{section}

\newcommand{\spt}[1]{\mbox{\normalfont spt}\Parans{#1}}
\newcommand{\sptBar}[2]{\overline{\mbox{\normalfont spt}}_{#1}\Parans{#2}}

\newcommand{\Parans}[1]{\left(#1\right)}

\newcommand{\PieceTwo}[4]
{
	\left\{
   	\begin{array}{ll}
      	#1 & #3 \\
       	#2 & #4
     	\end{array}
	\right.
}
\newcommand{\aqprod}[3]{\Parans{#1;#2}_{#3}}

\newcommand{\SB}{{\mbox{\rm S}}}

\author{CHRIS JENNINGS-SHAFFER}
\address{Department of Mathematics, University of Florida\\
Gainesville, Florida 32611, USA
\endgraf cjenningsshaffer@ufl.edu}



\title{Another Proof of Two Modulo 3 Congruences and Another SPT Crank for 
the Number of Smallest Parts in Overpartitions with even smallest part}

\allowdisplaybreaks
\begin{document}

\allowdisplaybreaks

\begin{abstract}
By considering the $M_2$-rank of an overpartition as well as a residual crank,
we give another combinatorial refinement of the congruences
$\sptBar{2}{3n}\equiv \sptBar{2}{3n+1}\equiv 0\pmod{3}$. Here $\sptBar{2}{n}$ is the total number
of occurrences of the smallest parts among the overpartitions of $n$ where the
smallest part is even and not overlined.
Our proof depends on Bailey's Lemma
and the rank difference formulas of Lovejoy and Osburn for the 
$M_2$-rank of an overpartition. This congruence, along with a modulo $5$
congruence, has previously been refined using the rank of an overpartition.
\end{abstract}

\maketitle

\section{Introduction and statement of Results}
\allowdisplaybreaks

We recall a partition of a positive integer $n$ is a non-increasing sequence
of positive integers that sum to $n$. As an example, the partitions of $4$
are $4$, $3+1$, $2+2$, $2+1+1$, and $1+1+1+1$. Similar to this is the idea
of an overpartition. An overpartition of $n$ is a partition of $n$ where the
first occurrence of a part may (or may not) be overlined. For example, the 
overpartitions of $4$ are $4$, $\overline{4}$, $3+1$, $3+\overline{1}$,
$\overline{3}+1$, $\overline{3}+\overline{1}$, $2+2$, 
$\overline{2}+2$, $2+1+1$, $2+\overline{1}+1$, $\overline{2}+1+1$,
$\overline{2}+\overline{1}+1$, $1+1+1+1$, and $\overline{1}+1+1+1$.

We have a weighted count on partitions, and overpartitions, given by counting a
partition by the number of times the smallest occurs. We use the convention  of
not counting overpartitions where the smallest part is overlined. We let
$\spt{n}$ denote the total number of occurrences of the smallest parts among 
the partitions of $n$. We let $\sptBar{}{n}$ denote the total number of 
occurrences of the smallest parts among the overpartitions of $n$ without smallest
part overlined. The function $\spt{n}$ was introduced by Andrews in 
\cite{Andrews} and the function $\sptBar{}{n}$ was introduced by Bringmann,
Lovejoy, and Osburn in \cite{BLO2}. Two restrictions of $\sptBar{}{n}$
are $\sptBar{1}{n}$ and $\sptBar{2}{n}$, where restrict to overpartitions
where the smallest part is odd and even respectively. We see
$\spt{4}=10$, $\sptBar{}{4}=13$, $\sptBar{1}{4}=10$, and $\sptBar{2}{4}=3$.

Similar to the work of Andrews, Garvan, and Liang in \cite{AGL} for
$\spt{n}$, in \cite{GarvanJennings} Garvan and the author gave combinatorial
refinements of congruences satisfied by $\sptBar{}{n}$, $\sptBar{1}{n}$, and
$\sptBar{2}{n}$. The idea is to introduce an extra variable into the generating
function of each spt-function to get a crank type statistic. This statistic can
then be shown in certain cases to equally split up the numbers $\sptBar{}{n}$, 
$\sptBar{1}{n}$, and $\sptBar{2}{n}$ based on the residue class of the
statistic. We explain this in more detail shortly.

For $\sptBar{2}{n}$ we have the congruences
\begin{align}\label{EqTheMod3Congruences1}
	\sptBar{2}{3n} \equiv 0 \pmod{3},
	\\\label{EqTheMod3Congruences2}
	\sptBar{2}{3n+1} \equiv 0 \pmod{3},
	\\\label{EqTheMod5Congruences}
	\sptBar{2}{5n+3} \equiv 0 \pmod{5}
.
\end{align}
In this paper we give another proof of the modulo $3$ congruences.

To start, by summing according to the smallest part, we find a generating
function for $\sptBar{2}{n}$ to be given by
\begin{align}
	\sum_{n=1}^\infty \sptBar{2}{n}q^n
	&=
	\sum_{n=1}^\infty \frac{q^{2n} \aqprod{-q^{2n+1}}{q}{\infty}}
	{(1-q^{2n})^2\aqprod{q^{2n+1}}{q}{\infty}}.
\end{align}
Here we use the standard product notation,
\begin{align}
	\aqprod{a}{q}{n} &= \prod_{j=0}^{n-1}(1-aq^j)
	,\\
	\aqprod{a_1,a_2,\dots,a_k}{q}{n}
	&=
	\aqprod{a_1}{q}{n}\aqprod{a_2}{q}{n}\dots\aqprod{a_k}{q}{n}
	,\\
	\aqprod{a}{q}{\infty} &= \prod_{j=0}^{\infty}(1-aq^j)
	,\\
	\aqprod{a_1,a_2,\dots,a_k}{q}{\infty}
	&=
	\aqprod{a_1}{q}{\infty}\aqprod{a_2}{q}{\infty}\dots\aqprod{a_k}{q}{\infty}
.
\end{align}

In \cite{GarvanJennings} we considered the two variable generalization given
by 
\begin{align}
	\overline{\mbox{S}}_2(z,q)
	&=
	\sum_{n=1}^\infty
		\frac{q^{2n}\aqprod{-q^{2n+1}}{q}{\infty}\aqprod{q^{2n+1}}{q}{\infty}}
			{\aqprod{zq^{2n}}{q}{\infty}\aqprod{z^{-1}q^{2n}}{q}{\infty}}
.
\end{align}
We note setting $z=1$ gives the generating function for $\sptBar{2}{n}$.
It turns out $\overline{\mbox{S}}_2(z,q)$ can be expressed in terms of the Dyson rank of an
overpartition and a residual crank from \cite{BLO2}. In \cite{LO1}
Lovejoy and Osburn determined difference formulas for the Dyson rank
of an overpartition, these formulas are essential to the proofs in
\cite{GarvanJennings}. In the same paper, we also used the difference formulas
for the $M_2$-rank of a partition without repeated odd parts, these formulas
were also determined by Lovejoy and Osburn in \cite{LO2}. Lovejoy
and Osburn also found difference formulas for the $M_2$-rank
of an overpartition \cite{LO3}, but we did not use these formulas in 
\cite{GarvanJennings}.

In \cite{CJS1} the
author gave higher order generalizations of $\sptBar{}{n}$ and
$\sptBar{2}{n}$ and noted that one could use the $M_2$-rank and another
residual crank from \cite{BLO2}, as used in that paper in working with $\sptBar{2}{n}$,
to explain the modulo $3$ congruences for $\sptBar{2}{n}$. Here we give the proof.
In this paper we instead use
\begin{align}
	\SB(z,q) &= \sum_{n=1}^\infty
	\frac{q^{2n}\aqprod{-q^{2n+1}}{q}{\infty}\aqprod{q^{2n+1}}{q}{\infty}}
		{\aqprod{zq^{2n},z^{-1}q^{2n}}{q^2}{\infty}\aqprod{q^{2n+1}}{q^2}{\infty}^2}
	\\
	&=
	\sum_{n=1}^\infty \sum_{m=-\infty}^\infty N_{\SB}(m,n)z^mq^n
.
\end{align}
Again setting $z=1$ gives the generating function for $\sptBar{2}{n}$. This is
not the same series $S(z,q)$ used in \cite{GarvanJennings}, however we do not 
want to overly complicate matters with additional notation.

For a positive integer $t$ we let
\begin{align}
	N_{\SB}(k,t,n) = \sum_{m\equiv k \pmod{t}}N_{\SB}(m,n).	
\end{align}
We then have
\begin{align}
	\sptBar{2}{n} &= \sum_{m=-\infty}^\infty N_{\SB}(m,n)
	= \sum_{k=0}^{t-1} N_{\SB}(k,t,n)
.
\end{align}
Additionally we see if $\zeta$ is a $t^{th}$ root of unity, then
\begin{align}
	\SB(\zeta,q)
	&=
	\sum_{n=1}^\infty \left(\sum_{k=0}^{t-1} N_{\SB}(k,t,n)\zeta^k\right)q^n
.
\end{align}
We consider when $\zeta=\zeta_3$ is a primitive third root of unity. Here the minimal
polynomial for $\zeta_3$ is $x^2+x+1$. If 
$N_{\SB}(0,t,N)+N_{\SB}(1,t,N)\zeta_3+N_{\SB}(2,t,N)\zeta_3^2=0$ then we must in fact have 
$N_{\SB}(0,t,N)=N_{\SB}(1,t,N)=N_{\SB}(2,t,N)$. That is to say, if
the coefficient of $q^N$ in $\SB(\zeta_3,q)$ is zero, then 
$\sptBar{2}{N} = 3\cdot N_{\SB}(0,t,N)$ and so 
$\sptBar{2}{N} \equiv 0 \pmod{3}$.

Our proof of $\sptBar{2}{3n}\equiv\sptBar{2}{3n+1}\equiv 0$ is to find the $3$-dissection of 
$\SB(\zeta_3,q)$ with the $q^{3n}$ and $q^{3n+1}$ terms being all zero.
This is the same idea that was used in \cite{GarvanJennings}, we are just using
$\SB(z,q)$ rather than $\overline{\mbox{S}}_2(z,q)$.

\begin{theorem}\label{SptBarCrankDissection}
\begin{align}
	\SB(\zeta_3,q)
	&=  A_0(q^3) + qA_1(q^3) + q^2A_2(q^3)
,
\end{align}
where
\begin{align}
	A_0(q) &= 0,	
	\\
	A_1(q) &= 0,	
	\\
	A_2(q) 
	&= 
	\frac{\aqprod{q^6}{q^6}{\infty}^4}
		{\aqprod{q^2}{q^2}{\infty}\aqprod{q^3}{q^3}{\infty}^2}	
	+\frac{2q\aqprod{-q^3}{q^3}{\infty}}{\aqprod{q^3}{q^3}{\infty}}
	\sum_{n=-\infty}^\infty \frac{(-1)^n q^{3n^2+6n}}{1-q^{6n+2}}
.	
\end{align}
\end{theorem}

We prove these formulas by relating $\SB(z,q)$ to a certain rank and crank
and using dissections of these related functions.
We recall the $M_2$-rank of an overpartition $\pi$ is given by
\begin{align}
	M_2\mbox{-rank} = \left\lceil\frac{l(\pi)}{2} \right\rceil - \#(\pi)
		+\#(\pi_o) - \chi(\pi)
,
\end{align}
where $l(\pi)$ is the largest part of $\pi$, $\#(\pi)$ is the number of
parts of $\pi$, $\#(\pi_o)$ is the number of non-overlined odd parts,
and $\chi(\pi)=1$ if the largest part of $\pi$ is odd and non-overlined 
and otherwise  $\chi(\pi)=0$. The $M_2$-rank was introduced by 
Lovejoy in \cite{LO3}. We let $\overline{N2}(m,n)$ denote the number of 
overpartitions of $n$ with $M_2$-rank $m$. Lovejoy found
the generating function for $\overline{N2}$ is given by
\begin{align}\label{RankDef}
	\sum_{n=0}^\infty\sum_{m=-\infty}^\infty \overline{N2}(m,n)z^mq^n
	&= 
	\sum_{n=0}^\infty \frac{\aqprod{-1}{q}{2n}q^n}
			{\aqprod{zq^2,z^{-1}q^2}{q^2}{n}}
	=
	\frac{\aqprod{-q}{q}{\infty}}{\aqprod{q}{q}{\infty}}
		\Parans{1+
			2\sum_{n=1}^\infty\frac{(1-z)(1-z^{-1})(-1)^nq^{n^2+2n}}
				{(1-zq^{2n})(1-z^{-1}q^{2n})}
		}
.
\end{align}

We also use a residual crank first defined \cite{BLO2}. We define
\begin{align}\label{CrankDef}
	\sum_{n=0}^\infty\sum_{m=-\infty}^\infty \overline{M2}(m,n)z^mq^n
	&= \frac{\aqprod{-q}{q}{\infty}\aqprod{q^2}{q^2}{\infty}}
		{\aqprod{q}{q^2}{\infty}\aqprod{zq^2}{q^2}{\infty}\aqprod{z^{-1}q^2}{q^2}{\infty}}.
\end{align}
The interpretation of this is as follows.
For an overpartition $\pi$ of $n$ we take 
the crank of the partition $\frac{\pi_e}{2}$ obtained by taking the 
subpartition $\pi_e$, of the even non-overlined parts of 
$\pi$, and halving each part of $\pi_e$.
Then $\overline{M2}(m,n)$ is the number of overpartitions $\pi$ of $n$  
and such that the partition $\frac{\pi_e}{2}$ has crank $m$.
However this interpretation fails when considering overpartitions
whose only even non-overlined parts are a single two, as the
corresponding interpretation of 
$\frac{\aqprod{q}{q}{\infty}}{\aqprod{zq,q/z}{q}{\infty}}$
as the generating function of the crank for ordinary partitions fails for
the partition of $1$.

\begin{theorem}\label{SptBar2RankCrank}
\begin{align}
	\SB(z,q)
	&=
	\frac{1}{(1-z)(1-z^{-1})}  
	\sum_{n=0}^\infty\sum_{m=-\infty}^\infty 
		\Parans{\overline{N2}(m,n)-\overline{M2}(m,n)}z^mq^n
.
\end{align}
\end{theorem}

Using the rank difference formulas derived by Lovejoy and Osburn in 
\cite{LO3}, we have
\begin{theorem}\label{DissectionForRank}
\begin{align}
	\sum_{n=0}^\infty\sum_{m=-\infty}^\infty
	\overline{N2}(m,n)\zeta_3^mq^n
	&= \overline{N2}_{0}(q^3) 
		+ q\overline{N2}_{1}(q^3) 
		+ q^2\overline{N2}_{2}(q^3)
,
\end{align}
where
\begin{align}
	\overline{N2}_{0}(q)
	&=
	\frac{\aqprod{-q}{q}{\infty}\aqprod{q^3}{q^3}{\infty}^2}
		{\aqprod{q}{q}{\infty}\aqprod{-q^3}{q^3}{\infty}^2}
	,
	\\
	\overline{N2}_{1}(q)
	&=
	\frac{2\aqprod{q^3}{q^3}{\infty}\aqprod{q^6}{q^6}{\infty}}
		{\aqprod{q}{q}{\infty}}
	,\\
	\overline{N2}_{2}(q)
	&=
	\frac{4\aqprod{q^6}{q^6}{\infty}^4}
		{\aqprod{q^2}{q^2}{\infty}\aqprod{q^3}{q^3}{\infty}^2}
	+
	\frac{6q\aqprod{-q^3}{q^3}{\infty}}{\aqprod{q^3}{q^3}{\infty}}
	\sum_{n=-\infty}^\infty \frac{(-1)^n q^{3n^2+6n}}{1-q^{6n+2}}
.
\end{align}
\end{theorem}

\begin{theorem}\label{M2CrankDissection}
\begin{align}
	\frac{\aqprod{-q}{q}{\infty}\aqprod{q^2}{q^2}{\infty}}
		{\aqprod{q}{q^2}{\infty}\aqprod{\zeta_3q^2,\zeta_3^{-1}q^2}{q^2}{\infty}}
	&=
	\sum_{n=0}^\infty\sum_{m=-\infty}^\infty
		\overline{M2}(m,n)\zeta_3^mq^n
	= \overline{M2}_{0}(q^3) 
		+ q\overline{M2}_{1}(q^3) 
		+ q^2\overline{M2}_{2}(q^3)
,
\end{align}
where
\begin{align}
	\overline{M2}_{0}(q)
	&=
	\frac{\aqprod{-q}{q}{\infty}\aqprod{q^3}{q^3}{\infty}^2}
		{\aqprod{q}{q}{\infty}\aqprod{-q^3}{q^3}{\infty}^2}
	,
	\\
	\overline{M2}_{1}(q)
	&=
	2\frac{\aqprod{q^3}{q^3}{\infty}\aqprod{q^6}{q^6}{\infty}}
		{\aqprod{q}{q}{\infty}}
	,\\
	\overline{M2}_{2}(q)
	&=
	\frac{\aqprod{q^6}{q^6}{\infty}^4}
		{\aqprod{q^2}{q^2}{\infty}\aqprod{q^3}{q^3}{\infty}^2}
.
\end{align}
\end{theorem}

We see Theorem \ref{SptBarCrankDissection} follows from Theorems
\ref{SptBar2RankCrank}, \ref{DissectionForRank}, and \ref{M2CrankDissection}, noting
$\frac{1}{(1-\zeta_3)(1-\zeta_3^{-1})}=\frac{1}{3}$.
We give the proofs of Theorems \ref{SptBar2RankCrank} and 
\ref{M2CrankDissection} in the next section. In Section 3 we give brief
combinatorial interpretations of the coefficients $N_{\SB}(m,n)$, in 
particular they are non-negative, and in Section 4 we end with a few remarks.

\section{The Proofs}

\begin{proof}[Proof of Theorem \ref{SptBar2RankCrank}]

We recall a pair of sequences  $(\alpha_n,\beta_n)$,
is a Bailey pair relative to $(a,q)$ if
\begin{align}
	\beta_n &= \sum_{r=0}^n \frac{\alpha_r}
		{\aqprod{q}{q}{n-r}\aqprod{aq}{q}{n+r}}
.
\end{align}
A limiting case of Bailey's Lemma gives for a Bailey pair
$(\alpha_n,\beta_n)$ that
\begin{align}
	\sum_{n=0}^\infty
	\aqprod{\rho_1,\rho_2}{q}{n} \Parans{\frac{aq}{\rho_1\rho_2}}^n  \beta_n
	&=
	\frac{\aqprod{aq/\rho_1,aq/\rho_2}{q}{\infty}}
		{\aqprod{aq,aq/\rho_1\rho_2}{q}{\infty}}
	\sum_{n=0}^\infty
	\frac{\aqprod{\rho_1,\rho_2}{q}{n} (\frac{aq}{\rho_1\rho_2})^n \alpha_n}
		{\aqprod{aq/\rho_1,aq/\rho_2}{q}{n}}
.
\end{align}

As in \cite{CJS1} the Bailey pair connecting the $M_2$-rank of
an overpartition and the residual crank is
\begin{align}
	\alpha_n &= \PieceTwo{1}{(-1)^n2q^{n^2}}{n=0}{n\ge 1}
	,\\
	\beta_n &= \frac{\aqprod{q}{q^2}{n}^2}{\aqprod{q^2}{q^2}{2n}}
. 
\end{align}
This is a Bailey pair with respect to $(1,q^2)$.

We note that
\begin{align}
	\frac{\aqprod{q^2}{q^2}{\infty}}{\aqprod{z,z^{-1}}{q^2}{\infty}\aqprod{q}{q^2}{\infty}^2}
	\cdot
	\frac{\aqprod{zq^2,z^{-1}q^2}{q^2}{\infty}}{\aqprod{q^2}{q^2}{\infty}^2}
	&=
	\frac{\aqprod{q^2}{q^2}{\infty}}{(1-z)(1-z^{-1})\aqprod{q}{q}{\infty}^2}
	=
	\frac{\aqprod{-q}{q}{\infty}}{(1-z)(1-z^{-1})\aqprod{q}{q}{\infty}}
	.
\end{align}
With this Bailey pair we have
\begin{align}
	\SB(z,q)
	&=
	\sum_{n=1}^\infty 
	\frac{q^{2n}\aqprod{-q^{2n+1}}{q}{\infty}\aqprod{q^{2n+1}}{q}{\infty}}
		{\aqprod{zq^{2n},z^{-1}q^{2n}}{q^2}{\infty}\aqprod{q^{2n+1}}{q^2}{\infty}^2}
	\nonumber\\
	&=
	\sum_{n=1}^\infty 
	\frac{q^{2n}\aqprod{q^{4n+2}}{q^2}{\infty}}
		{\aqprod{zq^{2n},z^{-1}q^{2n}}{q^2}{\infty}\aqprod{q^{2n+1}}{q^2}{\infty}^2}
	\nonumber\\
	&=
	\frac{\aqprod{q^2}{q^2}{\infty}}{\aqprod{z,z^{-1}}{q^2}{\infty}\aqprod{q}{q^2}{\infty}^2}
	\sum_{n=1}^\infty 
		q^{2n}\aqprod{z,z^{-1}}{q^2}{n}\beta_n
	\nonumber\\
	&=
	\frac{\aqprod{q^2}{q^2}{\infty}}{\aqprod{z,z^{-1}}{q^2}{\infty}\aqprod{q}{q^2}{\infty}^2}
	\sum_{n=0}^\infty 
		q^{2n}\aqprod{z,z^{-1}}{q^2}{n}\beta_n
	-
	\frac{\aqprod{q^2}{q^2}{\infty}}{\aqprod{z,z^{-1}}{q^2}{\infty}\aqprod{q}{q^2}{\infty}^2}
	\nonumber\\
	&=
	\frac{\aqprod{-q}{q}{\infty}}{(1-z)(1-z^{-1})\aqprod{q}{q}{\infty}}
	\sum_{n=0}^\infty 
	\frac{q^{2n}\aqprod{z,z^{-1}}{q^2}{n}  \alpha_n}
		{\aqprod{zq^2,z^{-1}q^2}{q^2}{n}}
	-
	\frac{\aqprod{q^2}{q^2}{\infty}}{\aqprod{z,z^{-1}}{q^2}{\infty}\aqprod{q}{q^2}{\infty}^2}
	\nonumber\\
	&=
	\frac{\aqprod{-q}{q}{\infty}}{(1-z)(1-z^{-1})\aqprod{q}{q}{\infty}}
	\sum_{n=0}^\infty 
	\frac{q^{2n}(1-z)(1-z^{-1})\alpha_n}
		{(1-zq^{2n})(1-z^{-1}q^{2n})}
	-
	\frac{\aqprod{-q}{q^2}{\infty}\aqprod{q^2}{q^2}{\infty}}
		{\aqprod{q}{q^2}{\infty}\aqprod{z,z^{-1}}{q^2}{\infty}}
	\nonumber\\
	&=
	\frac{\aqprod{-q}{q}{\infty}}{(1-z)(1-z^{-1})\aqprod{q}{q}{\infty}}
	\left(
		1+
		2\sum_{n=1}^\infty 
		\frac{(1-z)(1-z^{-1})(-1)^nq^{n^2+2n}}
			{(1-zq^{2n})(1-z^{-1}q^{2n})}
	\right)
	-
	\frac{\aqprod{-q}{q^2}{\infty}\aqprod{q^2}{q^2}{\infty}}
		{\aqprod{q}{q^2}{\infty}\aqprod{z,z^{-1}}{q^2}{\infty}}
.
\end{align}
By equations (\ref{RankDef}) and (\ref{CrankDef}) we then have
\begin{align}
	\SB(z,q)
	&=
	\frac{1}{(1-z)(1-z^{-1})}  
	\sum_{n=0}^\infty\sum_{m=-\infty}^\infty 
		\Parans{\overline{N2}(m,n)-\overline{M2}(m,n)}z^mq^n
.
\end{align}
\end{proof}

\begin{proof}[Proof of Theorem \ref{M2CrankDissection}]
We begin by noting
\begin{align}
	\frac{\aqprod{-q}{q}{\infty}\aqprod{q^2}{q^2}{\infty}}
		{\aqprod{q}{q^2}{\infty}\aqprod{\zeta_3q^2,\zeta_3^{-1}q^2}{q^2}{\infty}}
	&=
	\frac{\aqprod{q^2}{q^2}{\infty}^2}{\aqprod{q}{q^2}{\infty}^2\aqprod{q^6}{q^6}{\infty}}
.
\end{align}
By Gauss we have
\begin{align}
	\frac{\aqprod{q^2}{q^2}{\infty}}{\aqprod{q}{q^2}{\infty}}
	&=
	\sum_{n=0}^\infty q^{n(n+1)/2}
	=
	\frac{1}{2}	\sum_{n=-\infty}^\infty q^{n(n+1)/2}
.
\end{align}

By the Jacobi Triple Product Identity we then have
\begin{align}
	\frac{\aqprod{q^2}{q^2}{\infty}}{\aqprod{q}{q^2}{\infty}}
	&=
	\frac{1}{2}\sum_{k=0}^2\sum_{n=-\infty}^\infty q^{(3n+k)(3n+k+1)/2}	
	\nonumber\\
	&=
	\frac{1}{2}\sum_{n=-\infty}^\infty q^{(9n^2+3n)/2}	
	+\frac{1}{2}\sum_{n=-\infty}^\infty q^{(9n^2+9n)/2}	
	+\frac{1}{2}\sum_{n=-\infty}^\infty q^{(9n^2+15n)/2}	
	\nonumber\\
	&=
	\frac{1}{2}\aqprod{-q^6,-q^3,q^9}{q^9}{\infty}
	+\frac{1}{2}q\aqprod{-1,-q^9,q^9}{q^9}{\infty}
	+\frac{1}{2}\aqprod{-q^{-3},-q^{12},q^9}{q^9}{\infty}
	\nonumber\\
	&=
	\aqprod{-q^6,-q^3,q^9}{q^9}{\infty}
	+q\aqprod{-q^9,-q^9,q^9}{q^9}{\infty}
.	
\end{align}

Using the above to expand $\frac{\aqprod{q^2}{q^2}{\infty}^2}{\aqprod{q}{q^2}{\infty}^2}$
and dividing by $\aqprod{q^6}{q^6}{\infty}$ then gives
\begin{align}
	\overline{M2}_{0}(q)
	&=
	\frac{\aqprod{-q,-q^2,q^3}{q^3}{\infty}^2}{\aqprod{q^2}{q^2}{\infty}}
	,
	\\
	\overline{M2}_{1}(q)
	&=
	2\frac{\aqprod{-q,-q^2,-q^3,-q^3,q^3,q^3}{q^3}{\infty}}{\aqprod{q^2}{q^2}{\infty}}
	\\
	\overline{M2}_{2}(q)
	&=
	\frac{\aqprod{-q^3,-q^3,q^3}{q^3}{\infty}^2}{\aqprod{q^2}{q^2}{\infty}}
.
\end{align}
However these products easily reduce to those in the statement of the Theorem
\ref{M2CrankDissection}.
\end{proof}

\section{Combinatorial Interpretations}

We see $N_{\SB}(m,n)$ can be interpreted in terms of vector partitions.
We let $\overline{V} = \mathcal{D}\times\mathcal{P}\times\mathcal{P}\times\mathcal{D}$, 
where $\mathcal{P}$ denotes the set of all partitions and $\mathcal{D}$ 
denotes the set of all partitions into distinct parts. For a partition 
$\pi$ we let $s(\pi)$ denote the smallest part of $\pi$ 
(with the convention that the empty partition has smallest part $\infty$),
$\#(\pi)$ the number of parts in $\pi$,
$\#(\pi_e)$ the number of even parts in $\pi$,
and $|\pi|$ the sum of the parts of $\pi$. 
For $\vec{\pi} = (\pi^1,\pi^2,\pi^3,\pi^4) \in\overline{V}$, we define the weight
$\omega(\vec{\pi}) = (-1)^{\#(\pi^1)-1 }$, the 
$\mbox{crank}(\vec{\pi}) = \#(\pi^2_e) - \#(\pi^3_e)$,
and the norm $|\vec{\pi}|=|\pi^1|+|\pi^2|+|\pi^3|+|\pi^4|$. We say $\vec{\pi}$ is 
a vector partition of $n$ if $|\vec{\pi}|=n$.

By writing the summands of $\SB(z,q)$ as
\begin{align}
	&q^{2n}\aqprod{q^{2n+1}}{q}{\infty}
	\cdot
	\frac{1}{\aqprod{zq^{2n}}{q^2}{\infty}\aqprod{q^{2n+1}}{q^2}{\infty}}
	\cdot
	\frac{1}{\aqprod{z^{-1}q^{2n}}{q^2}{\infty}\aqprod{q^{2n+1}}{q^2}{\infty}}
	\cdot
	\aqprod{q^{2n+1}}{q}{\infty}
,
\end{align}
we see $N_{\SB}(m,n)$ is the number of vector partitions
$(\pi^1,\pi^2,\pi^3,\pi^4)$ from $\overline{V}$ of $n$ with the additional constraints
that $\pi^1$ is non-empty,
$s(\pi^1)\le s(\pi^2)$, $s(\pi^1)\le s(\pi^3)$, $s(\pi^1)< s(\pi^4)$,
and $s(\pi^1)$ is even, but counted with the weight $\omega$.

However, this interpretation hides the fact that each $N_{\SB}(m,n)$ is
non-negative, as we are counting with a weight that may be negative. 
We can also interpret $N_{\SB}(m,n)$ in terms of partition pairs.
This interpretation makes the non-negativity clear. Using the $q$-binomial theorem 
we have
\begin{align}
	&\sum_{n=1}^\infty
	\frac{q^{2n}\aqprod{q^{4n+2}}{q^2}{\infty}}
		{\aqprod{zq^{2n},z^{-1}q^{2n}}{q^2}{\infty}\aqprod{q^{2n+1}}{q^2}{\infty}^2}
	\nonumber\\
	&=
	\sum_{n=1}^\infty
	\frac{q^{2n}}{\aqprod{zq^{2n}}{q^2}{\infty}\aqprod{q^{2n+1}}{q^2}{\infty}^2}
	\sum_{k=0}^\infty 
	\frac{\aqprod{zq^{2n+2}}{q^2}{k} z^{-k}q^{2nk}}
		{\aqprod{q^2}{q^2}{k}}
	\nonumber\\
	&=
	\sum_{n=1}^\infty
	\frac{q^{2n}}{\aqprod{zq^{2n}}{q^2}{\infty}\aqprod{q^{2n+1}}{q^2}{\infty}^2}
	+
	\sum_{n=1}^\infty\sum_{k=1}^\infty
	\frac{z^{-k}q^{2n+2nk}}
		{(1-zq^{2n})\aqprod{zq^{2n+2k+2}}{q^2}{\infty}\aqprod{q^2}{q^2}{k}\aqprod{q^{2n+1}}{q^2}{\infty}^2}
	\nonumber\\
	&=
	\sum_{n=1}^\infty
	\frac{q^{2n}}{\aqprod{zq^{2n}}{q^2}{\infty}\aqprod{q^{2n+1}}{q^2}{\infty}^2}
	\nonumber\\&\quad
		+
		\sum_{n=1}^\infty\sum_{k=1}^\infty
		\frac{q^{2n}\aqprod{q^2}{q^2}{n}}
			{(1-zq^{2n})\aqprod{q^2}{q^2}{n+k}\aqprod{zq^{2n+2k+2}}{q^2}{\infty}\aqprod{q^{2n+1}}{q^2}{\infty}}
		\cdot
		\frac{z^{-k}q^{2nk}\aqprod{q^2}{q^2}{n+k}}
			{\aqprod{q^2}{q^2}{k}\aqprod{q^2}{q^2}{n}\aqprod{q^{2n+1}}{q^2}{\infty}}
	\nonumber\\
	&=
	\sum_{n=1}^\infty
	\frac{q^{2n}}{\aqprod{zq^{2n}}{q^2}{\infty}\aqprod{q^{2n+1}}{q^2}{\infty}^2}
	\nonumber\\\label{EqPairCrankSeries}&\quad
		+
		\sum_{n=1}^\infty\sum_{k=1}^\infty
		\frac{q^{2n}}
			{(1-zq^{2n})\aqprod{q^{2n+2}}{q^2}{k}\aqprod{zq^{2n+2k+2}}{q^2}{\infty}
			\aqprod{q^{2n+1}}{q^2}{\infty}}
		\cdot
		\frac{z^{-k}q^{2nk}\aqprod{q^2}{q^2}{n+k}}
			{\aqprod{q^2}{q^2}{k}\aqprod{q^2}{q^2}{n}\aqprod{q^{2n+1}}{q^2}{\infty}}
.
\end{align}

We can now give the partition pair interpretation. We let $PP_2$ denote the set of
partition pairs $(\pi^1,\pi^2)$ such that $\pi^1$ is non-empty, $s(\pi^1)$ is 
even, $s(\pi^1)\le s(\pi^2)$, and the even parts of $\pi^2$
are at most $2s(\pi^1)$. For such a partition pair we 
let $k(\pi^1,\pi^2)$ denote the number of even parts of $\pi^1$ that are either 
the smallest part or are larger than $s(\pi_1)+2\#(\pi^2_e)$. We note when 
$\pi^2$ contains no even parts that $k(\pi_1,\pi_2)$ reduces to $\#(\pi^1_e)$. 
We define a crank on the elements of $PP_2$ by
\begin{align}
	c(\pi_1,\pi_2)
	&=	
		k(\pi^1,\pi^2) - \#(\pi^2_e) 	 - 1	
.
\end{align}
We claim $N_{\SB}(m,n)$ is also the number of partitions pairs of $n$ from
$PP_2$ with $c(\pi^1,\pi^2)=m$.

For this we note the first series in (\ref{EqPairCrankSeries}) gives the cases 
when $\pi^2$ has no
even parts. The second series in (\ref{EqPairCrankSeries}) gives the cases when $\pi^2$ has even
parts, since
 $\frac{q^{2nk}\aqprod{q^2}{q^2}{n+k}}{\aqprod{q^2}{q^2}{k}\aqprod{q^2}{q^2}{n}}$
is the generating function for partitions into even parts with exactly $k$ parts
and each part between $2n$ and $4n$ (inclusive).

It may be possible to define a bijection from these partition pairs to
marked overpartitions with smallest part even, and through that determine
a crank defined on marked overpartitions. However, we do not pursue that here.

\section{Remarks}
While $\SB(\zeta_3,q)$ can be used to prove
$\sptBar{2}{3n}\equiv 0 \pmod{3}$ 
and $\sptBar{2}{3n+1}\equiv 0 \pmod{3}$, we cannot use $\SB(\zeta_5,q)$ to prove
$\sptBar{2}{5n+3}\equiv 0 \pmod{5}$. In particular we find the coefficient
of $q^8$ in $\SB(\zeta_5,q)$ to be
$z^3+z^2+3z+5+3z^{-1}+z^{-2}+z^{-3}$. That is to say,
$N_{\SB}(0,5,8) = 5$,
$N_{\SB}(1,5,8) = 3$,
$N_{\SB}(2,5,8) = 2$,
$N_{\SB}(3,5,8) = 2$, and
$N_{\SB}(4,5,8) = 3$.

However $\sptBar{2}{5n+3}\equiv 0 \pmod{5}$ does follow by considering
$\overline{\mbox{S}}_2(\zeta_5,q)$.
This can be compared with the rank of a partition explaining the congruences for
$p(5n+4)$ and $p(7n+5)$ but not $p(11n+6)$, whereas the crank of an partition
does explain all three.

\bibliographystyle{abbrv}
\bibliography{alternateSptBar2Mod3Ref}

\end{document}